\documentclass[11pt,a4paper]{article}

\usepackage{amsmath}
\usepackage{amssymb,latexsym}
\usepackage{amsthm}
\usepackage[margin=0.35in,format=hang,font=small,labelfont=bf,justification=justified,singlelinecheck=false]{caption}
\usepackage{needspace}
\usepackage{enumerate}
\usepackage{tikz}
\usepackage[section]{placeins}
\usepackage[margin=2.5cm]{geometry}
\RequirePackage[hyperfootnotes=false,colorlinks=true]{hyperref}

\newenvironment{bullets} {\vspace{-9pt}\begin{itemize}\itemsep0pt} {\end{itemize}\vspace{-9pt}}

\newcommand{\+}{\hspace{0.07 em}}  
\newcommand{\nequiv}{\not\equiv}
\newcommand{\floor}[1]{\left\lfloor #1 \right\rfloor}

\newcommand{\liminfty}[1][n]{\lim\limits_{#1\rightarrow\infty}}
\newcommand{\limsupinfty}[1][n]{\limsup\limits_{#1\rightarrow\infty}}
\newcommand{\liminfinfty}[1][n]{\liminf\limits_{#1\rightarrow\infty}}
\newcommand{\bbT}{\mathbb{T}}
\newcommand{\bbZ}{\mathbb{Z}}
\newcommand{\nv}[1]{$#1$\rlap{$^{\dagger}$}}

\newcommand{\ov}[1]{$#1$}
\newcommand{\mv}[1]{$#1$\rlap{*}}
\newcommand{\pv}[1]{$#1$}

\newcommand{\plotgrid}[2]  
{
  \foreach \x in {1,...,#1} \draw[gray,very thin] (\x,1)--(\x,#2+1);
  \foreach \y in {1,...,#2} \draw[gray,very thin] (1,\y)--(#1+1,\y);
  \draw[thick,gray!75!white] (1,1) rectangle (#1+1,#2+1);
  \foreach \x in {1,...,#1}
  {
    \foreach \y in {1,...,#2} \fill[gray!75!black,radius=0.1] (\x,\y) circle;
  }
}

\DeclareMathOperator{\Cay}{Cay}

\newtheorem{theorem}{Theorem}
\newtheorem{corollary}[theorem]{Corollary}
\newtheorem{lemma}[theorem]{Lemma}

\newtheorem{observation}[theorem]{Observation}

\title{\textbf{Large circulant graphs \\
of fixed diameter and arbitrary degree}}

\author{
David Bevan\thanks{University of Strathclyde, Glasgow, UK}\\ \texttt{\small david.bevan@strath.ac.uk}
\and Grahame Erskine\thanks{Open University, Milton Keynes, UK}\\ \texttt{\small grahame.erskine@open.ac.uk}
\and Robert Lewis\footnotemark[2]\\ \texttt{\small robert.lewis@open.ac.uk}
}

\date{}

\begin{document}
\maketitle
\let\thefootnote\relax\footnote{MSC: 05C25,05C35. Keywords: degree-diameter problem, Cayley graphs, circulant graphs, sumsets.}
\begin{abstract}\noindent
We consider the degree-diameter problem for undirected and directed circulant graphs. 
To date, attempts to generate families of large circulant graphs of arbitrary degree for a given diameter have concentrated mainly on the diameter 2 case. 
We present a direct product construction yielding improved bounds for small diameters and introduce a new general technique for ``stitching'' 
together circulant graphs which enables us to improve the current best known asymptotic orders for every diameter. 
As an application, we use our constructions in the directed case to obtain upper bounds on the minimum size of a subset 
$A$ of a cyclic group of order $n$ such that the $k$-fold sumset $kA$ is equal to the whole group. 
We also present a revised table of largest known circulant graphs of small degree and diameter.
\end{abstract}

\section{Introduction}
The goal of the \emph{degree-diameter problem} is to identify the largest possible number $n(d,k)$ of vertices in a graph having diameter $k$ and maximum degree $d$.
This paper considers the problem for the restricted category of circulant graphs, which we view as Cayley graphs of cyclic groups.
We consider both undirected and directed versions of the problem in this paper.
For a history and more complete summary of the degree-diameter problem, see the survey paper by Miller and \v{S}ir\'a\v{n}~\cite{miller2005moore}.

All groups considered in this paper are Abelian (indeed cyclic) and hence we use additive notation for the group operation.
With this convention we define a \emph{Cayley graph} as follows.

Let $G$ be an Abelian group and $S\subseteq G$ a subset such that $0\notin S$.
Then the \emph{Cayley graph} $\Cay(G,S)$ has the elements of $G$ as its vertex set and each vertex $g$ has an edge to $g+s$ for each $s\in S$.
The following properties of $\Cay(G,S)$ are immediate from the definition.
\begin{bullets}
 \item $\Cay(G,S)$ has order $|G|$ and is a regular graph of degree $|S|$.
 \item $\Cay(G,S)$ has diameter at most $k$ if and only if every element of $G$ can be expressed as a sum of no more than $k$ elements of $S$.
 \item $\Cay(G,S)$ is an undirected graph if $S=-S$; otherwise it is a directed graph.
\end{bullets}

A \emph{circulant graph} is a Cayley graph of a cyclic group, and we use these terms interchangeably.

\vbox{
Throughout the paper we use the following notation.
\begin{bullets}
  \item $CC(d,k)$ is the largest order of an undirected circulant graph with degree $d$ and diameter $k$.
  \item $DCC(d,k)$ is the largest order of a directed circulant graph with degree $d$ and diameter $k$.
\end{bullets}
}
For a given diameter $k$, we are interested in determining the asymptotics of $CC(d,k)$ and $DCC(d,k)$ as the degree $d$ tends to infinity. We make use of the following limits.
\begin{bullets}
  \item $L^-_C(k) = \liminfinfty[d] \,CC(d,k)/d^k$;  ~$L^+_C(k) = \limsupinfty[d] \,CC(d,k)/d^k$.
  \item $L^-_D(k) = \liminfinfty[d] \,DCC(d,k)/d^k$; ~$L^+_D(k) = \limsupinfty[d] \,DCC(d,k)/d^k$.
\end{bullets}

We begin with some trivial bounds on $L^-$ and $L^+$.
The following asymptotic upper bound is easily obtained; see for example the survey paper~\cite{miller2005moore}:
\begin{observation}[Trivial upper bound]\label{obsTrivialUB}
  $L^+_C(k) \leq L^+_D(k) \leq \frac{1}{k!}$.
\end{observation}
For a lower bound, consider $\bbZ_{r^k}$ with generators $\{h r^\ell : |h|\leq \floor{\frac{r}{2}}\!,\, 0\leq \ell<k\}$:
\begin{observation}[Trivial lower bound]\label{obsTrivialLB}
  $L^-_D(k) \geq L^-_C(k) \geq \frac{1}{k^k}$.
\end{observation}

In this paper, we present constructions which yield, for each $k\geq2$, lower bounds on $L^-_C(k)$ and $L^-_D(k)$ that are greater than the trivial $1/k^k$ bound.
No such bounds were known previously.
Our results include the following (see Corollary~\ref{corOrders}):
\begin{bullets}
  \item For any diameter $k\geq2$ and any degree $d$ large enough, $CC(d,k)>\left(1.14775\+\frac{d}{k}\right)^k$.
  \item For any diameter $k$ that is a multiple of 5 or sufficiently large, and any degree $d$ large enough, $CC(d,k)>\left(1.20431\+\frac{d}{k}\right)^k$.
  \item For any diameter $k\geq2$ and any degree $d$ large enough, $DCC(d,k)>\left(1.22474\+\frac{d}{k}\right)^k$.
  \item For any diameter $k$ that is a multiple of 6 or sufficiently large, and any degree $d$ large enough, $DCC(d,k)>\left(1.27378\+\frac{d}{k}\right)^k$.
\end{bullets}
We also
deduce a result concerning sumsets covering $\mathbb{Z}_n$,
and use our techniques to construct a revised
table of the largest known circulant graphs of small degree and diameter.

For larger diameters, the trivial bounds become numerically small, and the ratio between the upper and lower bound becomes arbitrarily large.
Therefore, in order more easily to assess the success of our constructions, we make use of the following measure which records improvement over the trivial lower bound.

Let $R^-_C(k) = k\+L^-_C(k)^{1/k}$, and define $R^+_C(k)$, $R^-_D(k)$ and $R^+_D(k)$ analogously.
Thus, $R^-_C(k)\geq1$, with equality if the trivial lower bound is approached asymptotically for large degrees. For each $k$, these $R$ values thus provide a useful indication of the success of our constructions in exceeding the trivial lower bound.
In Section~\ref{sectStitching}, we show how to construct a cyclic Cayley graph from two smaller ones in such a way that the $R$ values are preserved.

The $R$ values are bounded above by $R_{\max}(k)=k(k!)^{-1/k}$. Using the asymptotic version of Stirling's approximation, $\log k!\sim k\log k-k$, we see that as the diameter tends to infinity,
\[
1 \;\leq\; \liminfinfty[k]R^-_C(k) \;\leq\; \liminfinfty[k]R^+_C(k) \;\leq\; e ,
\]
and similarly for $R^-_D(k)$ and $R^+_D(k)$.

In the next section, we extend a result of Vetr\'ik~\cite{Vetrik2014} to deduce new lower bounds for $L^-_C(2)$ and $R^-_C(2)$.
In Section~\ref{sectDirectProd}, we describe a direct product construction and use it to build large cyclic Cayley graphs 
of small diameter and arbitrarily large degree.
We also prove that this construction is unable to yield values that exceed the trivial lower bound for large diameter.
However, in Section~\ref{sectStitching}, we demonstrate a method of building a circulant graph by ``stitching'' together
two smaller ones, and show how the application of this method to the constructions from Section~\ref{sectDirectProd} enables us to exceed the trivial lower bound for every diameter.

Section~\ref{sectSumsets} contains an application of our constructions
to obtain upper bounds on the minimum size of a set $A\subseteq\bbZ_n$
such that the $k$-fold sumset $kA$ is equal to $\bbZ_n$.
We conclude, in Section~\ref{sectTables}, by presenting a revised table of the largest known circulant graphs of small degree and diameter, including a number of new largest orders resulting from our constructions.

\section{Diameter 2 bounds for all large degrees}
Much of the study of this problem to date has concentrated on the diameter 2 undirected case.
In this instance, the trivial lower bound on $L^-_C(2)$ is 1/4 and the trivial upper bound on $L^+_C(2)$ is 1/2.
Vetr\'ik~\cite{Vetrik2014} (building on Macbeth, {\v{S}}iagiov{\'a} \& {\v{S}}ir{\'a}{\v{n}}~\cite{macbeth2012cayley})
presents a construction that proves that
$L^+_C(2) \geq \frac{13}{36}\approx 0.36111$, and thus $R^+_C(2)>1.20185$.

In this section, we begin by extending this result to yield bounds for $L^-_C(2)$ and $R^-_C(2)$.
This argument can also be found in Lewis~\cite{Lewis2015}. We reproduce it here for completeness, since we make use of the resulting bounds below.

Vetr\'ik's theorem applies only to values of the degree $d$ of the form $6p-2$, where $p$ is a prime such that $p\neq 13,p\nequiv 1\pmod{13}$.
We extend this result to give a slightly weaker bound valid for all sufficiently large values of $d$.
The strategy is as follows:
\begin{bullets}
\item Given a value of $d$, we select the largest prime $p$ in the allowable congruence classes such that $6p-2\leq d$.
\item We construct the graph of Vetr\'ik~\cite{Vetrik2014} using this value of $p$.
\item We add generators to the Vetr\'ik construction (and hence edges to the Cayley graph) to obtain a new graph of degree $d$ which still has diameter 2.
\end{bullets}
Note that the graphs in the Vetr\'ik construction always have even order and hence we may obtain an odd degree $d$ by adding the unique element of order 2
to the generator set.

Success of this method relies on being able to find a prime $p$ sufficiently close to the optimal value so that we need only add asymptotically few edges
to our graph. We use recent results of Cullinan \& Hajir~\cite{Cullinan2012} following Ramar{\'e} \& Rumely~\cite{Ramare1996}.

\begin{lemma}[Cullinan \& Hajir~\cite{Cullinan2012}, Ramar{\'e} \& Rumely~\cite{Ramare1996}]\label{thmCHPrimeGaps}
Let $\delta=0.004049$.
For any $x_0\geq 10^{100}$ there exists a prime $p\equiv 2\pmod{13}$ in the interval $[x_0,x_0+\delta x_0]$.
\end{lemma}
\begin{proof}
We use the method of Cullinan and Hajir~\cite[Theorem~1]{Cullinan2012}. 
This method begins by using the tables of Ramar{\'e} and Rumely~\cite{Ramare1996} to find a value $\epsilon$ corresponding to $k=13,a=2,x_0=10^{100}$.
Following the proof of Cullinan and Hajir~\cite[Theorem~1]{Cullinan2012}, if $\delta>\frac{2\epsilon}{1-\epsilon}$ it follows that there must exist a 
prime $p\equiv 2\pmod{13}$ in the interval $[x_0,x_0+\delta x_0]$.
From Table 1, Ramar{\'e} and Rumely~\cite{Ramare1996} we find $\epsilon=0.002020$ and hence $\delta=0.004049$ will suffice.
\end{proof}

Our improved bound for circulant graphs of diameter 2 follows:
\begin{theorem}[{see~\cite[Theorem~6]{Lewis2015}}]\label{thmVetrikPrimeGaps}
$L^-_C(2) > 0.35820$, and hence $R^-_C(2) > 1.19700$.
\end{theorem}
\begin{proof}
Let $\delta=0.004049$ and let $d>10^{101}$. We seek the largest prime $p\equiv 2\pmod{13}$ such that $6p-2\leq d$.
By the result of Lemma~\ref{thmCHPrimeGaps}, there exists such a $p$ with $p\geq (d+2)/6(1+\delta)$.
Let $d'=6p-2$. Then by the result of Vetr\'ik~\cite{Vetrik2014} we can construct a circulant graph of degree $d'$, diameter 2 and order $n=\frac{13}{36}(d'+2)(d'-4)$.
We can add $d-d'$ generators to this construction to obtain a graph of degree $d$ and diameter 2, with the same order $n$.

Since $n=\frac{13}{36}d^2/(1+\delta)^2+O(d)\approx 0.358204 d^2+O(d)$ the result follows.
\end{proof}

\section{Direct product constructions for small diameters}\label{sectDirectProd}
In this section, we construct large undirected circulant graphs of diameters $k=3, 4, 5$ and arbitrary large degree.
We also construct large directed circulant graphs of diameters $k=2,\ldots,9$ and arbitrary large degree.
We then prove that the approach used is insufficient to yield values that exceed the trivial lower bound for large diameter.

\subsection{Preliminaries}
The diameter 2 constructions of Macbeth, {\v{S}}iagiov{\'a} \& {\v{S}}ir{\'a}{\v{n}} and of Vetr\'ik both have the form
$F_p^+\times F_p^*\times\bbZ_w$ for some fixed $w$ and variable $p$,
where $F_p^+$ and $F_p^*$ are the additive and multiplicative groups of the Galois field $GF(p)$.
Thus the first two components of their constructions are very tightly coupled, and this coupling is a key to their success.
However, a significant limitation of this method is that it is only applicable in the diameter 2 case.

In contrast, the constructions considered here have components that are as loosely coupled as possible.
For diameter $k$, they have the form $\bbZ_{r_1}\times\bbZ_{r_2}\times\ldots\times\bbZ_{r_k}\times\bbZ_w$ for some fixed $w$ and variable pairwise coprime $r_i$.
This gives us greater flexibility, especially in terms of the diameters we can achieve.
The price for this is that we lose the inherent structure of the finite field, which consequently places limits on the bounds we can achieve.

The constructions in this section make use of the following result concerning the representation of each element of the cyclic group $\bbT=\bbZ_r\times\bbZ_s$ ($r$ and $s$ coprime)
as the sum of a small multiple of the element $(1,1)$ and a small multiple of another element $(u,v)$.
It can be helpful to envisage $\bbT$ as a group of vectors on the $r\times s$ discrete torus.
\begin{lemma}
\label{lemmaFairyChess}
Let $u$, $d$, $s$ and $m$ be positive integers with $s>1$ and coprime to $m d$.
Let $v=u+d$.
Suppose $s\geq m v(u-1)$.
Then, for every element $(x,y)$ of $\bbT=\bbZ_{s+m d}\times\bbZ_s$,
there exist nonnegative integers $h<s+m v$ and $\ell<s-m(u-1)$
such that $(x,y)=h(1,1)+\ell(u,v)$.
\end{lemma}
Observe that the construction ensures that $(s+mv)(1,1)=m(u,v)$.
Figure~\ref{figFairyChess} illustrates the case with parameters $u=2$, $v=5$, $s=11$, $m=2$.
\begin{figure}[ht]
$$
  \begin{tikzpicture}[scale=0.42]
    \draw [red!50!gray,ultra thick] (1,1) -- (3,6);
    \draw [red!30!white,ultra thick] (3,6) -- (5.4,12);
    \draw [red!30!white,ultra thick] (5.4,1) -- (9.8,12);
    \draw [red!30!white,ultra thick] (9.8,1) -- (14.2,12);
    \draw [red!30!white,ultra thick] (14.2,1) -- (17,8);
    \draw [blue!50!black,ultra thick] (1,1) -- (2,2);
    \draw [blue!30!white,ultra thick] (2,2) -- (12,12);
    \draw [blue!30!white,ultra thick] (12,1) -- (18,7);
    \draw [blue!30!white,ultra thick] (1,7) -- (4,10);
    \plotgrid{17}{11}
    \foreach \x in {1,...,11} \fill[blue!50!black,radius=0.21] (\x,\x) circle;
    \foreach \x in {1,...,6} \fill[blue!50!black,radius=0.21] (\x+11,\x) circle;
    \foreach \x in {1,...,4} \fill[blue!50!black,radius=0.21] (\x,\x+6) circle;
    \draw[red!50!black,radius=0.2,very thick] (1,1) circle;
    \draw[red!50!black,radius=0.2,very thick] (3,6) circle;
    \draw[red!50!black,radius=0.2,very thick] (5,11) circle;
    \draw[red!50!black,radius=0.2,very thick] (7,5) circle;
    \draw[red!50!black,radius=0.2,very thick] (9,10) circle;
    \draw[red!50!black,radius=0.2,very thick] (11,4) circle;
    \draw[red!50!black,radius=0.2,very thick] (13,9) circle;
    \draw[red!50!black,radius=0.2,very thick] (15,3) circle;
    \draw[red!50!black,radius=0.2,very thick] (17,8) circle;
  \end{tikzpicture}
$$
\caption{Every element of $\bbZ_{17}\times\bbZ_{11}$ is the sum of one of the 21 solid elements and one of the 9 circled elements.}\label{figFairyChess}
\end{figure}
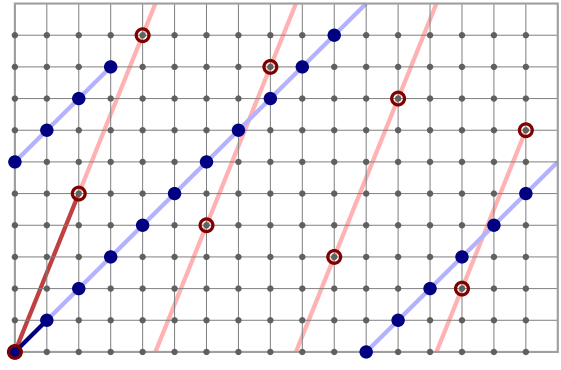
\begin{proof}
Let $t=s-m(u-1)$.
Since $s$ is coprime to $m d$, $(1,1)$ generates $\bbT$.
Hence, it suffices to show that, in the list $(0,0),\,(1,1),\,(2,2),\,\ldots$,
the gaps between members of $\{\ell(u,v):0\leq\ell<t\}$
are not ``too large''.

Specifically, we need to show that, for each nonnegative $\ell<t$, there is some
positive $h'\leq s+m v$ and nonnegative $\ell'<t$ such that $\ell(u,v)+h'(1,1)=\ell'(u,v)$.

There are two cases.
If $\ell<t-m$, then we can take $h'=s+m v$ and $\ell'=\ell+m$:
$$
\begin{array}{rcll}
  \ell(u,v) \:+\: (s+m v)(1,1) & = & (\ell u+s+m u+m d,         & \!\!\! \ell v+s+m v) \\[2pt]
                               & = & (\ell u\phantom{{}+s}+m u, & \!\!\! \ell v\phantom{{}+s}+m v) \\[2pt]
                               & = & (\ell+m)(u,v) .
\end{array}
$$
If $\ell\geq t-m$, then we can take $h'=m u v$ and $\ell'=\ell+m-t=\ell+m u-s$:
$$
\begin{array}{rcll}
  \ell(u,v) \:+\: m u v(1,1) & = & (\ell u+m u^2+m u d,            & \!\!\! \ell v+m u v) \\[2pt]
                             & = & (\ell u+m u^2+m u d - u(s+m d), & \!\!\! \ell v+m u v - v s) \\[2pt]
                             & = & (\ell+m u-s)(u,v) .
\end{array}
$$
The requirement that $m u v\leq s+m v$ is clearly equivalent to the condition on $s$ in the statement of the lemma.
\end{proof}

In our direct product constructions, we make use of Lemma~\ref{lemmaFairyChess} via the following crucial lemma.
Our strategy is to construct a cyclic group of the form $\bbT=\bbZ_{r_1}\times\bbZ_{r_2}\times\ldots\times\bbZ_{r_k}$ such that $r_1>r_2>\ldots>r_k>1$,
and for each pair $\bbZ_{r_i}\times\bbZ_{r_j}$ for $i<j$ we will bring Lemma~\ref{lemmaFairyChess} to bear.
In the notation of that lemma we will set $u=i$, $d=j-i$, $s=r_j$, $s+md=r_i$. 
The conditions of Lemma~\ref{lemmaFairyChessCor} below are designed to ensure that for each pair $i,j$ we can find $m=m_{i,j}$ to make this work.

\begin{lemma}\label{lemmaFairyChessCor}
  Let $k>1$ and let $r_1>r_2>\ldots>r_k$ be pairwise coprime integers greater than 1 such that $r_i$ is coprime to $i$ for $1\leq i\leq k$.
	Suppose that for each $i,j$ with $1\leq i<j\leq k$ there exists a positive integer $m_{i,j}$ such that:
  \begin{bullets}
    \item $r_i-r_j= m_{i,j}(j-i)$ 
    \item $r_j\geq m_{i,j}(i-1)\+j$.
  \end{bullets}
	
  Let $\bbT=\bbZ_{r_1}\times\bbZ_{r_2}\times\ldots\times\bbZ_{r_k}$.
	Let $\mathbf{o}=(1,1,\ldots,1)$, $\mathbf{u}=(1,2,\ldots,k)$
  and, for each $i$, $\mathbf{e}_i=(0,\ldots,1,\ldots,0)$
  be elements of $\bbT$,
  where only the $i$th coordinate of $\mathbf{e}_i$ is 1,
  and let
  the set $A$ consist of these $k+2$ elements.

  Let $c_\mathbf{o}=\max\limits_{i<j} (r_j+j\+m_{i,j})$, $c_\mathbf{u}=r_1$,
  and for each $i$, $c_{\mathbf{e}_i}=r_i$.

  Then, for every element $\mathbf{x}$ of $\bbT$ and every $k$-element subset $S$ of $A$, there exist nonnegative integers $h_\mathbf{s}<c_\mathbf{s}$ for each
  $\mathbf{s}\in S$, such that $\mathbf{x} = \sum_\mathbf{s} h_\mathbf{s}\+\mathbf{s}$.
\end{lemma}
\begin{proof}
  There are four cases.
  If $S$ contains neither $\mathbf{o}$ nor $\mathbf{u}$, the result follows trivially.

  If $S$ contains $\mathbf{o}$ but not $\mathbf{u}$, omitting $\mathbf{e}_i$, then we can choose $h_\mathbf{o}$ to be the $i$th coordinate of $\mathbf{x}$.
	Note that, as required, $c_\mathbf{o}\geq r_2+2(r_1-r_2)=r_1+(r_1-r_2)> r_i$ for all $i$.

  If $S$ contains $\mathbf{u}$ but not $\mathbf{o}$, omitting $\mathbf{e}_i$, then, since $i$ and $r_i$ are coprime,
	we can choose $h_\mathbf{u}$ such that $i\+h_\mathbf{u}\!\!\pmod{r_i}$ is the $i$th coordinate of $\mathbf{x}$.

  Finally, if $S$ contains both $\mathbf{o}$ and $\mathbf{u}$, omitting $\mathbf{e}_i$ and $\mathbf{e}_j$, then
  we can choose $h_\mathbf{o}$ and $h_\mathbf{u}$
  by applying Lemma~\ref{lemmaFairyChess} to $\bbZ_{r_i}\times\bbZ_{r_j}$ with $(u,v)=(i,j)$.
\end{proof}
We note that the conditions of Lemma~\ref{lemmaFairyChessCor} imply that at most one of the $r_i$ can be even, and if $k\geq 4$ then all $r_i$ must be odd.

\subsection{Undirected constructions}
We can use Lemma~\ref{lemmaFairyChessCor} to construct undirected circulant graphs of any diameter by means of the following theorem.
\begin{theorem}\label{thmUndir}
Let $w$ and $k$ be positive integers and suppose that there exist sets $B$ and $T$ of positive integers with the following properties:
\begin{bullets}
	\item $B=\{b_1,\ldots,b_{k+2}\}$ has cardinality $k+2$ and the property that every element of $\bbZ_w$ can be expressed as the
sum of exactly $k$ distinct elements of $B\cup-B$, no two of which are inverses.
\item $T=\{r_1,r_2,\ldots,r_k\}$ has cardinality $k$ and the properties that all its elements are coprime to $w$, 
and satisfies the requirements of Lemma~\ref{lemmaFairyChessCor}, i.e. for each $i<j$:
	\begin{bullets}
        \item $r_i>r_j$
		\item $\gcd(r_i,r_j)=1$
		\item $\gcd(r_i,i)=1$
		\item There is a positive integer $m_{i,j}$ such that $r_i-r_j=m_{i,j}(j-i)$ and $r_j\geq m_{i,j}(i-1)j$.
	\end{bullets}
\end{bullets}

\vspace{1.5ex}
Let $c_\mathbf{o}=\displaystyle\max_{i<j}(r_j+jm_{i,j})$ and $c_\mathbf{u}=r_1$ as in Lemma~\ref{lemmaFairyChessCor}.

Then there exists an undirected circulant graph of order $\displaystyle w\prod_{i=1}^k r_i$, 

degree at most $\displaystyle 2\left(\sum_{i=1}^k r_i+c_o+c_u\right)$ and diameter $k$.
\end{theorem}
\begin{proof}
Let $\bbT=\bbZ_{r_1}\times\bbZ_{r_2}\times\ldots\times\bbZ_{r_k}\times\bbZ_w$. Then $\bbT$ is a cyclic group since all its factors have coprime orders.

Let $X$ be the generating set consisting of the following elements.
\begin{bullets}
  \item $(x,0,0,\ldots,0,\pm b_1)$, $x\in\bbZ_{r_1}$
  \item $(0,x,0,\ldots,0,\pm b_2)$, $x\in\bbZ_{r_2}$
	\\[2pt] $\vdots$
  \item $(0,0,\ldots,0,x,\pm b_k)$, $x\in\bbZ_{r_k}$
  \item $\pm(x,x,\ldots,x,x,b_{k+1})$, $0\leq x<c_\mathbf{o}$
  \item $\pm(x,2x,\ldots,(k-1)x,kx,b_{k+2})$, $0\leq x<c_\mathbf{u}$
\end{bullets}
Then by construction and by Lemma~\ref{lemmaFairyChessCor}, every element of $\bbT$ is the sum of at most $k$ elements of $X$.
Since $\displaystyle |\bbT|=w\prod_{i=1}^k r_i$ and $\displaystyle |X|=2\left(\sum_{i=1}^k r_i+c_o+c_u\right)$, the result follows.
\end{proof}
For small diameters this technique results in the following asymptotic bounds.
\begin{theorem}\label{thmUndirSmall}
For diameters $k=3,4,5$, we have the following lower bounds:
\begin{enumerate}[(a)]
	\item $L_C^+(3)\geq\frac{57}{1000}$ and $L_C^-(3)\geq\frac{7}{125}$, so $R_C^+(3)>1.15455$ and $R_C^-(3)>1.14775$.
	\item $L_C^+(4)\geq L_C^-(4)\geq\frac{25}{3456}$, so $R_C^+(4)\geq R_C^-(4)>1.16654$.
	\item $L_C^+(5)\geq L_C^-(5)\geq\frac{109}{134456}$, so $R_C^+(5)\geq R_C^-(5)>1.20431$.
\end{enumerate}
\end{theorem}
\begin{proof}
Given a diameter $k$, the strategy is to find an optimal value of $w$ which admits a set $B$ satisfying the conditions of Theorem~\ref{thmUndir}.
We then seek an infinite family of positive integers $q$ and a set $\Delta=\{\delta_1,\delta_2,\ldots,\delta_{k-1}\}$ such that for each of our
values of $q$, the set $T=\{q,q-\delta_1,\ldots,q-\delta_{k-1}\}$ satisfies the conditions of the theorem. We illustrate for $k=3$.

To prove \emph{(a)} we take $w=57$ and $B=\{1,2,7,8,27\}$. It is easily checked that every element of $\bbZ_{57}$ is the sum of three distinct
elements of $B\cup-B$, no two of which are inverses. Now we let $\Delta=\{4,6\}$. For any $q\geq 17,q\equiv 5\pmod{6},q\not\equiv 0,4,6\pmod{19}$
it is straightforward to verify that the set $T=\{q,q-4,q-6\}$ satisfies the conditions of Theorem~\ref{thmUndir}.
In the notation of Lemma~\ref{lemmaFairyChessCor}, we have $c_\mathbf{o}=q+4$.

Taking a generating set $X$ as defined in Theorem~\ref{thmUndir} we may construct a circulant graph of diameter 3, degree $d=|X|=10q-12$ and order
$57q(q-4)(q-6)=\frac{57}{1000}(d+12)(d-28)(d-48)$.

We can do this for an infinite number of values of $q$, and hence for an infinite number of values of $d=10q-12$ we have
\[CC(d,3)\geq\frac{57}{1000}(d+12)(d-28)(d-48)\]
This yields $L_C^+(3)\geq\frac{57}{1000}$. Now we need to consider $L_C^-(3)$.
The strategy will be to try to add ``few'' edges to our graphs to cover all possible degrees.
Observe that we can use this construction for any $q\equiv 17\pmod{114}$ and hence for any $d\equiv 158\pmod{1140}$.
Given any arbitrary \emph{even} degree $d$, we can therefore find some $d'$ no smaller than $d-1140$ for which the construction works.
We can then add $d-d'$ generators to our graph to obtain a graph of the same order, degree $d$ and diameter 3.

However our graphs always have odd order, and so we are unable to obtain an odd degree graph by this method.
To get round this problem we may use $w=56$, $B=\{1,2,7,14,15\}$, $\Delta=\{2,4\}$ and $c_\mathbf{o}=q+2$. 
Again it is easy to check that the relevant conditions are satisfied for any $q\geq 15$ such that $q\equiv 3,5\pmod{6}$ and $q\equiv 1,3,5,6\pmod{7}$.
Then for $d=10q-8$ we can construct a graph of order $\frac{7}{25}(d+8)(d-12)(d-32)$, degree $d$ and diameter 3.
We can do this for any $q\equiv 15\pmod{42}$ and hence for any $d\equiv 142\pmod{420}$.
So given any arbitrary degree $d$, we can therefore find some $d'$ no smaller than $d-420$ for which the construction works,
and then add $d-d'$ generators to our graph to obtain a graph of the same order and diameter 3.
(Since our graphs now have even order it is possible to add an odd number of generators.)
Since the number of added generators is bounded above (by 419), the order of the graph is $\frac{7}{125}d^3+O(d^2)$.
Result \emph{(a)} for $L_C^-(3)$ follows.

For \emph{(b)} and \emph{(c)} we adopt a similar method, except that in both cases the graphs used have even order and so our bounds on $L^+$ and $L^-$ are equal.
For brevity we show only the relevant sets in the construction, summarised as follows.

\emph{(b)} ($k=4$) -- Take $w=150$, $B=\{1,7,16,26,41,61\}$ and $\Delta=\{6,8,12\}$ so $c_\mathbf{o}=q+6$.
Then for $q\geq 49$, $q\equiv 19\pmod {30}$ and $d=12q-40$, we have
\[CC(d,4)\geq\frac{25}{3456}(d+40)(d-32)(d-56)(d-104).\]

\emph{(c)} ($k=5$) -- Take $w=436$, $B=\{1,15,43,48,77,109,152\}$ and $\Delta=\{0,4,10,12,$ $16\}$ so $c_\mathbf{o}=q+8$.
Then for $q\geq 77$, $q\equiv 5\pmod{6}$, $q\nequiv 0,1\pmod{5}$, $q\nequiv 0,4,10,12,16\pmod{109}$ and $d=14q-68$, we have
\[CC(d,5)\geq\frac{109}{134456}(d+68)(d+12)(d-72)(d-100)(d-156).\qedhere\]
\end{proof}

\subsection{Directed constructions}
An analogous method yields directed circulant graphs via the following theorem.
\begin{theorem}\label{thmDir}
Let $w$ and $k$ be positive integers and suppose that there exist sets $B$ and $T$ of non-negative integers with the following properties:
\begin{bullets}
	\item $B=\{0,b_2,\ldots,b_{k+2}\}$ has cardinality $k+2$ and the property that every element of $\bbZ_w$ can be expressed as the
sum of exactly $k$ distinct elements of $B$.
\item $T=\{r_1,r_2,\ldots,r_k\}$ has cardinality $k$ and the properties that all its elements are coprime to $w$, 
and it satisfies the requirements of Lemma~\ref{lemmaFairyChessCor}, i.e. for each $i<j$:
	\begin{enumerate}[(a)]
    \item $r_i>r_j$
		\item $\gcd(r_i,r_j)=1$
		\item $\gcd(r_i,i)=1$
		\item There is a positive integer $m_{i,j}$ such that $r_i-r_j=m_{i,j}(j-i)$ and $r_j\geq m_{i,j}(i-1)j$.
	\end{enumerate}
\end{bullets}

\vspace{1.5ex}
Let $c_\mathbf{o}=\displaystyle\max_{i<j}(r_j+jm_{i,j})$ and $c_\mathbf{u}=r_1$ as in Lemma~\ref{lemmaFairyChessCor}.

Then we may construct a directed circulant graph of order $\displaystyle w\prod_{i=1}^k r_i$, degree at most $\displaystyle \sum_{i=1}^k r_i+c_o+c_u-1$
and diameter $k$.
\end{theorem}
\begin{proof}
Let $\bbT=\bbZ_{r_1}\times\bbZ_{r_2}\times\ldots\times\bbZ_{r_k}\times\bbZ_w$. Then $\bbT$ is a cyclic group since all its factors have coprime orders.

Let $X$ be the generating set consisting of the following elements.
\begin{bullets}
  \item $(x,0,0,\ldots,0,0)$, $x\in\bbZ_{r_1}\setminus\{0\}$
  \item $(0,x,0,\ldots,0,b_2)$, $x\in\bbZ_{r_2}$
	\\[2pt] $\vdots$
  \item $(0,0,\ldots,0,x,b_k)$, $x\in\bbZ_{r_k}$
  \item $(x,x,\ldots,x,x,b_{k+1})$, $0\leq x<c_\mathbf{o}$
  \item $(x,2x,\ldots,(k-1)x,kx,b_{k+2})$, $0\leq x<c_\mathbf{u}$
\end{bullets}
Then by construction and by Lemma~\ref{lemmaFairyChessCor}, every element of $\bbT$ is the sum of at most $k$ elements of $X$.
Since $\displaystyle |\bbT|=w\prod_{i=1}^k r_i$ and $\displaystyle |X|=\sum_{i=1}^k r_i+c_o+c_u-1$, the result follows.
\end{proof}

For small diameters this technique results in the following asymptotic bounds.
\begin{theorem}\label{thmDirSmall}
For diameters $k=2,\ldots,9$, we have the following lower bounds on $L_D^-(k)$ and $R_D^-(k)$
\begin{enumerate}[(a)]
	\item $L_D^-(2)\geq\frac{3}{8}$, so $R_D^-(2)>1.22474$.
	\item $L_D^-(3)\geq\frac{9}{125}$, so $R_D^-(3)>1.24805$.
	\item $L_D^-(4)\geq\frac{13}{1296}$, so $R_D^-(4)>1.26588$.
	\item $L_D^-(5)\geq\frac{17}{16807}$, so $R_D^-(5)>1.25881$.
	\item $L_D^-(6)\geq\frac{3}{32768}$, so $R_D^-(6)>1.27378$.
	\item $L_D^-(7)\geq\frac{10}{1594323}$, so $R_D^-(7)>1.26436$.
	\item $L_D^-(8)\geq\frac{9}{25000000}$, so $R_D^-(8)>1.25206$.
	\item $L_D^-(9)\geq\frac{42}{2357947691}$, so $R_D^-(9)>1.23939$.
\end{enumerate}
\end{theorem}
\begin{proof}
The method is exactly the same as the proof of Theorem~\ref{thmUndirSmall} and we summarise as follows.

\emph{(a)} ($k=2$) -- Take $w=6$, $B=\{0,1,2,4\}$ and $\Delta=\{2\}$ so $c_\mathbf{o}=q+2$.
Then for $q\geq 7$, $q\equiv 1\pmod {6}$ and $d=4q-1$, we have
\[DCC(d,2)\geq\frac{3}{8}(d+1)(d-7).\]

\emph{(b)} ($k=3$) -- Take $w=9$, $B=\{0,1,2,3,6\}$ and $\Delta=\{4,6\}$ so $c_\mathbf{o}=q+4$.
Then for $q\geq 17$, $q\equiv 5\pmod {6}$ and $d=5q-7$, we have
\[DCC(d,3)\geq\frac{9}{125}(d+7)(d-13)(d-23).\]

\emph{(c)} ($k=4$) -- Take $w=13$, $B=\{0,1,3,5,7,8\}$ and $\Delta=\{2,4,6\}$ so $c_\mathbf{o}=q+2$.
Then for $q\geq 23$, $q\equiv 5\pmod {6}$, $q\not\equiv 0,2,4,6\pmod{13}$ and $d=6q-11$, we have
\[DCC(d,4)\geq\frac{13}{1296}(d+11)(d-1)(d-13)(d-25).\]

\emph{(d)} ($k=5$) -- Take $w=17$, $B=\{0,1,2,3,4,8,13\}$ and $\Delta=\{4,10,12,16\}$ so $c_\mathbf{o}=q+8$.
Then for $q\geq 77$, $q\equiv 5\pmod {6}$, $q\not\equiv 0,1\pmod{5}$, $q\not\equiv 0,4,10,12,16\pmod{17}$ and $d=7q-35$, we have
\[DCC(d,5)\geq\frac{17}{16807}(d+35)(d+7)(d-35)(d-49)(d-77).\]

\emph{(e)} ($k=6$) -- Take $w=24$, $B=\{0,1,2,4,8,13,18,22\}$ and $\Delta=\{6,12,18,24,30\}$ so $c_\mathbf{o}=q+6$.
Then for $q\geq 181$, $q\equiv 1,5\pmod{6}$, $q\not\equiv 0,4\pmod {5}$ and $d=8q-85$, we have
\[DCC(d,6)\geq\frac{3}{32768}(d+85)(d+37)(d-11)(d-59)(d-107)(d-155).\]

\emph{(f)} ($k=7$) -- Take $w=30$, $B=\{0,1,2,6,9,12,16,17,18\}$ and $\Delta=\{0,2,6,18,20,$ $30,42\}$ so $c_\mathbf{o}=q+42$.
Then for $q\geq 529$, $q\equiv 1\pmod{6}$, $q\equiv 4\pmod{5}$, $q\nequiv 0,2,6\pmod{7}$, $q\nequiv 9\pmod {11}$ and $d=9q-77$, we have
\[DCC(d,7)\geq\frac{10}{1594323}(d+77)(d+59)(d+23)(d-85)(d-103)(d-193)(d-301).\]

\emph{(g)} ($k=8$) -- Take $w=36$, $B=\{0,1,2,3,6,12,19,20,27,33\}$ and $\Delta=\{0,6,12,$ $18,24,30,$ $36,42\}$ so $c_\mathbf{o}=q+6$.
Then for $q\geq 353$, $q\equiv 1,5\pmod{6}$, $q\equiv 3\pmod{5}$, $q\nequiv 0,1\pmod{7}$ and $d=10q-163$, we have
\begin{align*}
DCC(d,8)\geq\frac{9}{25000000}&(d+163)(d+103)(d+43)(d-17)\\
&(d-77)(d-137)(d-197)(d-257).
\end{align*}

\emph{(h)} ($k=9$) -- Take $w=42$, $B=\{0,1,2,3,4,9,16,20,26,30,37\}$ and $\Delta=\{0,2,6,$ $12,20,30,$ $42,56,72\}$ so $c_\mathbf{o}=q+72$.
Then for $q\geq 1093$, $q\equiv 1\pmod {6}$, $q\equiv 3,4\pmod {5}$, $q\equiv 1,3,4\pmod {7}$, $q\nequiv 1,6,9\pmod {11}$, 
$q\nequiv 4,7\pmod {13}$ and $d=11q-169$, we have
\begin{align*}
DCC(d,9)\geq\frac{42}{2357947691}&(d+169)(d+147)(d+103)(d+37)(d-51)\\
&(d-161)(d-293)(d-447)(d-623).\qedhere
\end{align*}

\end{proof}

\subsection{Limitations}
In~\cite{Lewis2015}, Lewis showed that an analogous class of constructions using finite fields to create graphs of diameter 2 is limited by the bound
$L^-_C(2)\leq \frac{3}{8}$. The constructions in this section have a similar limitation:

\begin{observation}\label{thmOverallLowerBound}
  Let $k$ be a positive integer.
  The direct product constructions of Theorems~\ref{thmUndir} and~\ref{thmDir} 
	can never yield a lower bound on $L^-_C(k)$ or $L^-_D(k)$ that exceeds $\frac{k+1}{2(k+2)^{k-1}}$.
\end{observation}
\begin{proof}
First we consider the undirected case.
Suppose the requirements of Theorem~\ref{thmUndir} hold and
for each $i=1,\ldots,k$, we have $r_i=q-a_i$, where $a_1<a_2<\ldots <a_k$.
Let $\bbT=\bbZ_{q-a_1}\times\ldots\times\bbZ_{q-a_k}\times\bbZ_w$ and
$X$ be its generating set as in the proof of Theorem~\ref{thmUndir}.

Since every element of $\bbZ_w$ is a sum of $k$ distinct elements of $B$, no pair of which are inverses,
we must have $w\leq \binom{k+2}{k}2^k=(k+1)(k+2)2^{k-1}$.

By the requirements of Lemma~\ref{lemmaFairyChessCor}, for any $i<j$, we have $m_{i,j}\leq r_i-r_j$ and $c_\mathbf{o}=\max\limits_{i<j} (r_j+j\+m_{i,j})$.
Hence, since $r_i=q-a_i$, we have $m_{i,j}\leq a_k-a_1$, and so $c_\mathbf{o}\leq q+ka_k$.

Thus $X$ is the generating set for a Cayley graph on $\bbT$ with diameter $k$, degree $d$ no greater than $2(k+2)q-2\sum_{i=1}^{k}{a_i} +2ka_k-2a_1$,
and order $n=w(q-a_1)(q-a_2)\ldots(q-a_k)$.

Hence,
$
n = \frac{w}{(2(k+2))^k}d^k+O(d^{k-1}) \leq \frac{(k+1)(k+2)2^{k-1}}{((2(k+2))^k}d^k+O(d^{k-1}) = \frac{k+1}{2(k+2)^{k-1}}d^k+O(d^{k-1}),
$
as required.

The directed case is analogous. We follow Theorem~\ref{thmDir} and its proof.
In this case, every
element of $\bbZ_w$ is the sum of $k$ distinct elements of $B$, so $w\leq \binom{k+2}{k}=(k+1)(k+2)/2$, and $X$ is the generating set for a Cayley graph on $\bbT$ with diameter $k$, degree $d\leq(k+2)q-\sum_{i=1}^{k}{a_i} +ka_k-a_1-1$,
and order $n=w(q-a_1)(q-a_2)\ldots(q-a_k)$.

Hence,
$
n = \frac{w}{(k+2)^k}d^k+O(d^{k-1}) \leq \frac{(k+1)(k+2)}{2(k+2)^k}d^k+O(d^{k-1}) = \frac{k+1}{2(k+2)^{k-1}}d^k+O(d^{k-1}).
$
\end{proof}

Observe that, in the limit,
$$
\liminfty[k] k\+\left( \frac{k+1}{2(k+2)^{k-1}}\right)^{\!1/k} =\; 1.
$$

As a consequence, these direct product constructions themselves can never yield an improvement on the trivial lower bound for the limiting value of $R^-_C(k)$ or $R^-_D(k)$.
However, it is possible to combine graphs of small diameter to produce larger graphs in such a way that we can improve
on the trivial lower bound in the limit as the diameter increases. The next section introduces this idea.

\section{A general graph product construction}\label{sectStitching}
The following theorem gives a simple way to combine two cyclic Cayley graphs to obtain a third cyclic Cayley graph. It is valid in both the directed and undirected cases.
\begin{theorem}\label{thmGraphProduct}
 Let $G_1$ and $G_2$ be two cyclic Cayley graphs of diameters $k_1$ and $k_2$, orders $n_1$ and $n_2$, and degrees $d_1$ and $d_2$ respectively.
 In the case of undirected graphs where $d_1$ and $d_2$ are both odd let $\delta=1$, otherwise $\delta=0$. In the directed case let $\delta=0$ always.
 Then there exists a cyclic Cayley graph with diameter $k_1+k_2$, degree at most $d_1+d_2+\delta$, and order $n_1n_2$.
\end{theorem}
\begin{proof}
  Let $S_1$ be the connection set of $G_1$ so that $|S_1|=d_1$ and similarly for $G_2$.
	For convenience we consider each $S_i$ to consist of elements within the interval $(-n_i/2,n_i/2]$.
	Let $G$ be the cyclic group $\bbZ_{n_1n_2}$ and consider the connection set $S'=n_2S_1\cup S_2$. Then $|S'|\leq n_1+n_2$.
	
	We now construct a connection set $S$ for the group $G$ such that the Cayley graph $\Cay(G,S)$ has diameter $k_1+k_2$.
	In the directed case we may simply take $S=S'$.
	In the undirected case we need to ensure that $S=-S$.
	If at least one of $d_1,d_2$ is even we may assume without loss of generality that $d_2$ is even and then we may again let $S=S'$ and $S=-S$ by construction.

	It remains to consider the undirected case when $d_1$ and $d_2$ are both odd (the case $\delta=1$).
	In that case we know $n_2/2\in S_2\subset S'$ and we let $S=S'\cup \{-n/2\}$ so that $S=-S$.
	
	It is then clear that the Cayley graph $\Cay(G,S)$ has degree at most $d_1+d_2+\delta$, diameter $k_1+k_2$ and order $n_1n_2$.
\end{proof}
We can use this ``stitching'' construction to obtain lower bounds on our $L$ and $R$ values for large diameters, given values for smaller diameters.
\begin{corollary}\label{corGraphProduct}
  If $L(k)$ is one of $L^-_C(k)$, $L^+_C(k)$, $L^-_D(k)$ or $L^+_D(k)$ and
  $R(k)$ is one of $R^-_C(k)$, $R^+_C(k)$, $R^-_D(k)$ or $R^+_D(k)$, then

  (a) $\displaystyle L(k_1+k_2) \;\geq\; \frac{L(k_1)\+L(k_2)\+k_1^{k_1}\+k_2^{k_2}}{(k_1+k_2)^{k_1+k_2}}$

  (b) $\displaystyle R(k_1+k_2) \;\geq\; \left( R(k_1)^{k_1} R(k_2)^{k_2} \right)^{\frac{1}{k_1+k_2}}$
\end{corollary}
\begin{proof} \emph{(a)} Let $d>1$. For $i=1,2$ we may construct graphs $\Gamma_i$ of diameter $k_i$, degree $k_i\+d$ and order $L(k_i)(k_i\+d)^{k_i}+o(d^{k_i})$.
Theorem~\ref{thmGraphProduct} yields a product graph of diameter $k_1+k_2$, degree at most $(k_1+k_2)d+1$
and order $L(k_1)L(k_2)k_1^{k_1}k_2^{k_2}d^{k_1+k_2}+o(d^{k_1+k_2})$.

Part \emph{(b)} follows by straightforward algebraic manipulation.
\end{proof}
In particular, we note that the stitching construction of Theorem~\ref{thmGraphProduct} preserves lower bounds on the
$R$ values: $R(mk)\geq R(k)$ for every positive integer $m$.

We may use this idea to obtain better bounds for some particular diameters; for example
we may improve on the undirected diameter 4 construction in Theorem~\ref{thmUndirSmall}:
\begin{corollary}\label{corCC4}
  (a) $L^+_C(4)\geq\frac{169}{20736}\approx0.0081501$, and hence $R^+_C(4)>1.20185$.

  (b) $L^-_C(4)>0.0080194$, and hence $R^-_C(4)>1.19700$.
\end{corollary}
\begin{proof}
  For \emph{(a)} we note $R_C^+(2)\geq\frac{13}{36}$ from Vetr\'ik~\cite{Vetrik2014} and apply Corollary~\ref{corGraphProduct} with $k_1=k_2=2$.
	For \emph{(b)} we use the same method starting with Theorem~\ref{thmVetrikPrimeGaps}.
\end{proof}

The stitching process of Theorem~\ref{thmGraphProduct} can be iterated to produce a construction for any desired diameter, and Corollary~\ref{corGraphProduct}
then gives us a lower bound for the $R$ values for that diameter.
We illustrate the results for small diameter $k$ in Table~\ref{tabR}.
As an indicator of progress we show also the largest possible value of $R$ for a particular $k$, given by $R_{\max}(k)=k(k!)^{-1/k}$.

\begin{table}[ht]
	\centering\footnotesize
	  \setlength{\tabcolsep}{0.3em}
    \renewcommand{\arraystretch}{1.2}
    \begin{tabular}{r|l|l|l|l|l|l|l|l|}
      \cline{2-9}
                   & \multicolumn{8}{|c|}{Diameter ($k$)} \\
      \cline{2-9}
                   & \multicolumn{1}{|c|}{2} & \multicolumn{1}{|c|}{3} & \multicolumn{1}{|c|}{4}
									 & \multicolumn{1}{|c|}{5} & \multicolumn{1}{|c|}{6} & \multicolumn{1}{|c|}{7} & \multicolumn{1}{|c|}{8} & \multicolumn{1}{|c|}{9}\\
      \cline{2-9}
      $R_{\max}(k) \approx$ & $1.41421$ &   $1.65096$ &   $1.80720$ &   $1.91926$ &   $2.00415$ &   $2.07100$ &   $2.12520$ &   $2.17016$ \\
      \cline{2-9}
      $R^+_C(k) >$ &      $1.20185^a$ & $1.15455^d$ & $1.20185^c$ & $1.20431^d$ & $1.20185^f$ & $1.20360^f$ & $1.20185^f$ & $1.20321^f$ \\
      \cline{2-9}
      $R^-_C(k) >$ &      $1.19700^b$ & $1.14775^d$ & $1.19700^c$ & $1.20431^d$ & $1.19700^f$ & $1.20222^f$ & $1.19700^f$ & $1.20105^f$ \\
      \cline{2-9}
      $R^-_D(k) >$ &      $1.22474^e$ & $1.24805^e$ & $1.26588^e$ & $1.25881^e$ & $1.27378^e$ & $1.26436^e$ & $1.26588^f$ & $1.26514^f$ \\
      \cline{2-9}
    \end{tabular}
  \caption{The best $R$ values for diameter $k\leq 9$\\
	$a$. Vetr\'ik~\cite{Vetrik2014};
	$b$. Theorem~\ref{thmVetrikPrimeGaps};
	$c$. Corollary~\ref{corCC4};\\
	$d$. Theorem~\ref{thmUndirSmall};
	$e$. Theorem~\ref{thmDirSmall};
	$f$. Corollary~\ref{corGraphProduct}
	}
  \label{tabR}
\end{table}

It is worth noting that the method of Corollary~\ref{corGraphProduct} may be used to produce values of $R$ which are larger than those achievable
from the direct product constructions of Section~\ref{sectDirectProd}.
For example, the limitations noted in Observation~\ref{thmOverallLowerBound} show that the maximum
possible value of $R^-_D(10)$ we could achieve using Theorem~\ref{thmDir} is approximately 1.26699. However, combining the results for diameters
4 and 6 in Table~\ref{tabR} yields $R^-_D(10)>1.27061$.

Next we use our previous results to show that $R$ is well-behaved in the limit.
\begin{theorem}\label{thmGraphProductR}
  Let $L(k)$ be one of $L^-_C(k)$, $L^+_C(k)$, $L^-_D(k)$ or $L^+_D(k)$, and let $R(k) = k\+L(k)^{1/k}$.
  The limit $R=\liminfty[k]R(k)$ exists and is equal to $\sup R(k)$.
\end{theorem}
\begin{proof}
$R(k)$ is bounded above (by $e$), so $s=\sup R(k)$ is finite.
  Hence, given $\varepsilon>0$, we can choose $k$ so that $s-R(k)<\varepsilon/2$.
  By Corollary~\ref{corGraphProduct}(b), $R(mk)\geq R(k)$ for every positive integer $m$.
  Moreover, for any fixed $j<k$, since $R(j)\geq 1$, we have
  $R(mk+j) \geq R(k)^{mk/(mk+j)} \geq R(k)^{m/(m+1)}$,
  which, by choosing $m$ large enough, can be made to
  differ from $R(k)$ by no more than $\varepsilon/2$.
\end{proof}
\begin{corollary}\label{corGraphProductR}~
\vspace*{-1em}
\begin{enumerate}[(a)]
  \item $\displaystyle\liminfty[k]R^-_C(k) \geq \frac{5\times 109^{1/5}}{7\times 2^{3/5}} > 1.20431$
  \item $\displaystyle\liminfty[k]R^-_D(k) \geq \frac{3^{7/6}}{2^{3/2}} > 1.27378$
\end{enumerate}
\end{corollary}
\begin{proof}
  We choose the largest entry in the relevant row in Table~\ref{tabR}.
	For \emph{(a)} we know from Theorem~\ref{thmUndirSmall} that $L^-_C(5)\geq \displaystyle\frac{109}{2^3\times 7^5}$.
  For \emph{(b)} we know from Theorem~\ref{thmDirSmall} that $L^-_D(6)\geq \displaystyle\frac{3}{2^{15}}$.
\end{proof}
We conclude this section by using the foregoing to derive new lower bounds for the maximum possible orders of circulant graphs of
given diameter and
sufficiently
large degree.
\Needspace*{3\baselineskip}
\begin{corollary}\label{corOrders}
~
\vspace*{-1em}
\begin{enumerate}[(a)]
  \item For any diameter $k\geq2$ and any degree $d$ large enough, $CC(d,k)>\left(1.14775\+\frac{d}{k}\right)^k$.
  \item For any diameter $k$ that is a multiple of 5 or sufficiently large, and any degree $d$ large enough, $CC(d,k)>\left(1.20431\+\frac{d}{k}\right)^k$.
  \item For any diameter $k\geq2$ and any degree $d$ large enough, $DCC(d,k)>\left(1.22474\+\frac{d}{k}\right)^k$.
  \item For any diameter $k$ that is a multiple of 6 or sufficiently large, and any degree $d$ large enough, $DCC(d,k)>\left(1.27378\+\frac{d}{k}\right)^k$.
\end{enumerate}
\end{corollary}
\begin{proof}~

\begin{enumerate}[(a)]
  \item From Theorem~\ref{thmGraphProductR} and Corollary~\ref{corGraphProduct}, we know that $R^-_C(k)$ always exceeds the smallest value in the $R^-_C$
	row of Table~\ref{tabR}, which is $1.14775$.
	\item For $k$ a multiple of 5, we know from Theorem~\ref{thmUndirSmall} and Corollary~\ref{corGraphProduct} that $R^-_C(k)>1.20431$.
	The result for sufficiently large $k$ follows from Corollary~\ref{corGraphProductR}.
	\item and (d) follow by using similar logic in the directed case.\qedhere
\end{enumerate}
\end{proof}
These represent significant improvements over the trivial bound of $\displaystyle\left(\frac{d}{k}\right)^k$.

\section{Sumsets covering \texorpdfstring{$\mathbb{Z}_n$}{Zn}}\label{sectSumsets}
Our constructions of directed circulant graphs can be used to obtain an upper bound on the minimum size,
$SS(n,k)$, of a set $A\subset\bbZ_n$ for which the sumset
\[
k\+A \;=\; \underbrace{A+A+\ldots+A}_k \;=\; \bbZ_n.
\]
The trivial bound is $SS(n,k)\leq kn^{1/k}$ which follows in the same way as the trivial lower bound for the directed circulant graph 
(see Observation~\ref{obsTrivialLB}).
Improvements to this trivial bound do not appear to have been investigated in the literature.

The idea is that, given $S\subseteq\bbZ_n$ such that $\Cay(\bbZ_n,S)$ has diameter $k$, if we let $A=S\cup\{0\}$ then $kA=\bbZ_n$.
Our constructions thus enable us to bound $SS(n,k)$ for fixed $k$ and infinitely many values of $n$.
For example, if we let $L^-_S(k)=\liminfinfty SS(n,k)/n^{1/k}$, then
the following new result for $k=2$ follows from Theorem~\ref{thmDirSmall}(a):
\begin{corollary}\label{corSS2}
  $L^-_S(2) \leq \sqrt{\frac{8}{3}} \approx 1.63299$.
\end{corollary}
More generally, Corollary~\ref{corGraphProductR} shows that for large enough $k$ and infinitely many values of $n$, $SS(n,k)$ is at least 21 percent smaller than the trivial bound:
\begin{corollary}\label{corSSR}
  $\liminfty[k]k^{-1}L^-_S(k) \leq \frac{2^{3/2}}{3^{7/6}} \approx 0.78506$.
\end{corollary}

\section{Largest graphs of small degree and diameter}\label{sectTables}
We can use the construction of Theorem~\ref{thmGraphProduct} to obtain large undirected circulant graphs for small degrees and diameters.
Recently in~\cite{Feria-Puron2015}, Feria-Puron, P\'erez-Ros\'es and Ryan published a table of largest known circulant graphs
with degree up to 16 and diameter up to 10.
Their method uses a construction based on graph Cartesian products which is somewhat similar to ours.
In contrast, however, Theorem~\ref{thmGraphProduct} does not in general result in a graph isomorphic to the Cartesian product of the constituents.
Furthermore, our construction does not require the constituent graph orders to be coprime, which allows more graphs to be constructed.

Using Theorem~\ref{thmGraphProduct} allowed us to improve many of the entries in the published table.
However, at the same time we developed a computer search which allows us to find circulant graphs of given degree, diameter and order.
It turns out that this search is able to find larger graphs (at least in the range $d\leq 16,k\leq 10$) than the Theorem~\ref{thmGraphProduct} method.
We therefore present a much improved table of largest known circulant graphs based on the outputs of this search.

In Table~\ref{tabBestCirc}, we show the largest known circulant graphs of degree $d \leq 16$ and diameter $k \leq 10$.
In Table~\ref{tab:search} we give a reduced generating set for each new record largest graph found by the search.
The computer search has been completed as an exhaustive search in the diameter 2 case up to degree 23,
and these results are included in Table~\ref{tab:search} for completeness.

\begin{table}[h]
\centering\small
\begin{tabular}[t]{|r|rrrrrrrrrr|}
\hline
$d\setminus k$ & 1 & 2 & 3 & 4 & 5 & 6 & 7 & 8 & 9 & 10\\
\hline
2 & \ov{3} & \ov{5} & \ov{7} & \ov{9} & \ov{11} & \ov{13} & \ov{15} & \ov{17} & \ov{19} & \ov{21}\\
3 & \ov{4} & \ov{8} & \ov{12} & \ov{16} & \ov{20} & \ov{24} & \ov{28} & \ov{32} & \ov{36} & \ov{40}\\
4 & \ov{5} & \ov{13} & \ov{25} & \ov{41} & \ov{61} & \ov{85} & \ov{113} & \ov{145} & \ov{181} & \ov{221}\\
5 & \ov{6} & \ov{16} & \ov{36} & \ov{64} & \ov{100} & \ov{144} & \ov{196} & \ov{256} & \ov{324} & \ov{400}\\
6 & \ov{7} & \ov{21} & \ov{55} & \ov{117} & \ov{203} & \ov{333} & \ov{515} & \ov{737} & \ov{1027} & \ov{1393}\\
7 & \ov{8} & \ov{26} & \ov{76} & \ov{160} & \ov{308} & \ov{536} & \ov{828} & \ov{1232} & \ov{1764} & \ov{2392}\\
8 & \ov{9} & \ov{35} & \ov{104} & \ov{248} & \ov{528} & \ov{984} & \ov{1712} & \ov{2768} & \ov{4280} & \ov{6320}\\
9 & \ov{10} & \ov{42} & \ov{130} & \ov{320} & \ov{700} & \ov{1416} & \ov{2548} & \ov{4304} & \ov{6804} & \ov{10320}\\
10 & \ov{11} & \ov{51} & \ov{177} & \ov{457} & \ov{1099} & \nv{2380} & \nv{4551} & \nv{8288} & \nv{14099} & \nv{22805}\\
11 & \ov{12} & \ov{56} & \ov{210} & \ov{576} & \nv{1428} & \nv{3200} & \nv{6652} & \nv{12416} & \nv{21572} & \nv{35880}\\
12 & \ov{13} & \ov{67} & \ov{275} & \nv{819} & \nv{2040} & \nv{4283} & \nv{8828} & \nv{16439} & \nv{29308} & \nv{51154}\\
13 & \ov{14} & \ov{80} & \ov{312} & \nv{970} & \nv{2548} & \nv{5598} & \nv{12176} & \nv{22198} & \nv{40720} & \nv{72608}\\
14 & \ov{15} & \ov{90} & \ov{381} & \nv{1229} & \nv{3244} & \nv{7815} & \nv{17389} & \nv{35929} & \nv{71748} & \nv{126109}\\
15 & \ov{16} & \ov{96} & \ov{448} & \nv{1420} & \nv{3980} & \nv{9860} & \nv{22584} & \nv{48408} & \nv{93804} & \nv{177302}\\
16 & \ov{17} & \ov{112} & \nv{518} & \nv{1717} & \nv{5024} & \nv{13380} & \nv{32731} & \nv{71731} & \nv{148385} & \nv{298105}\\
\hline
\end{tabular}
\captionsetup{justification=justified}
\caption{Largest known circulant graphs of degree $d\leq 16$ and diameter $k\leq 10$ \\
$\dagger$ new record largest value}
\label{tabBestCirc}
\end{table}

\begin{table}
\centering\footnotesize
\begin{tabular}[t]{|rrrl|}
\hline
$d$ & $k$ & Order & Generators\\
\hline
$6$ & $2$ & \mv{21} & $1,2,8$\\
$6$ & $3$ & \mv{55} & $1,5,21$\\
$6$ & $4$ & \mv{117} & $1,16,22$\\
$6$ & $5$ & \mv{203} & $1,7,57$\\
$6$ & $6$ & \mv{333} & $1,9,73$\\
$6$ & $7$ & \mv{515} & $1,46,56$\\
$6$ & $8$ & \mv{737} & $1,11,133$\\
$6$ & $9$ & \mv{1027} & $1,13,157$\\
$6$ & $10$ & \mv{1393} & $1,92,106$\\
$7$ & $2$ & \mv{26} & $1,2,8$\\
$7$ & $3$ & \mv{76} & $1,27,31$\\
$7$ & $4$ & \mv{160} & $1,5,31$\\
$7$ & $5$ & \mv{308} & $1,7,43$\\
$7$ & $6$ & \mv{536} & $1,231,239$\\
$7$ & $7$ & \mv{828} & $1,9,91$\\
$7$ & $8$ & \mv{1232} & $1,11,111$\\
$7$ & $9$ & \mv{1764} & $1,803,815$\\
$7$ & $10$ & \mv{2392} & $1,13,183$\\
$8$ & $2$ & \mv{35} & $1,6,7,10$\\
$8$ & $3$ & \mv{104} & $1,16,20,27$\\
$8$ & $4$ & \mv{248} & $1,61,72,76$\\
$8$ & $5$ & \mv{528} & $1,89,156,162$\\
$8$ & $6$ & \mv{984} & $1,163,348,354$\\
$8$ & $7$ & \mv{1712} & $1,215,608,616$\\
$8$ & $8$ & \pv{2768} & $1,345,1072,1080$\\
$8$ & $9$ & \pv{4280} & $1,429,1660,1670$\\
$8$ & $10$ & \pv{6320} & $1,631,2580,2590$\\
$9$ & $2$ & \mv{42} & $1,5,14,17$\\
$9$ & $3$ & \mv{130} & $1,8,14,47$\\
$9$ & $4$ & \mv{320} & $1,15,25,83$\\
$9$ & $5$ & \mv{700} & $1,5,197,223$\\
$9$ & $6$ & \pv{1416} & $1,7,575,611$\\
$9$ & $7$ & \pv{2548} & $1,7,521,571$\\
$9$ & $8$ & \pv{4304} & $1,9,1855,1919$\\
$9$ & $9$ & \pv{6804} & $1,9,1849,1931$\\
$9$ & $10$ & \pv{10320} & $1,11,4599,4699$\\
$10$ & $2$ & \mv{51} & $1,2,10,16,23$\\
$10$ & $3$ & \mv{177} & $1,12,19,27,87$\\
$10$ & $4$ & \mv{457} & $1,20,130,147,191$\\
$10$ & $5$ & \mv{1099} & $1,53,207,272,536$\\
$10$ & $6$ & \pv{2380} & $1,555,860,951,970$\\
$10$ & $7$ & \pv{4551} & $1,739,1178,1295,1301$\\
$10$ & $8$ & \pv{8288} & $1,987,2367,2534,3528$\\
$10$ & $9$ & \pv{14099} & $1,1440,3660,3668,6247$\\
$10$ & $10$ & \pv{22805} & $1,218,1970,6819,6827$\\
$11$ & $2$ & \mv{56} & $1,2,10,15,22$\\
$11$ & $3$ & \mv{210} & $1,49,59,84,89$\\
$11$ & $4$ & \mv{576} & $1,9,75,155,179$\\
$11$ & $5$ & \pv{1428} & $1,169,285,289,387$\\
$11$ & $6$ & \pv{3200} & $1,259,325,329,1229$\\
$11$ & $7$ & \pv{6652} & $1,107,647,2235,2769$\\
$11$ & $8$ & \pv{12416} & $1,145,863,4163,5177$\\
$11$ & $9$ & \pv{21572} & $1,663,6257,10003,10011$\\
\hline
\multicolumn{4}{c}{(Table continues on next page)}
\end{tabular}
\end{table}

\begin{table}
\centering\footnotesize
\begin{tabular}[t]{|rrrl|}
\hline
$d$ & $k$ & Order & Generators\\
\hline
$11$ & $10$ & \pv{35880} & $1,2209,5127,5135,12537$\\
$12$ & $2$ & \mv{67} & $1,2,3,13,21,30$\\
$12$ & $3$ & \mv{275} & $1,16,19,29,86,110$\\
$12$ & $4$ & \pv{819} & $7,26,119,143,377,385$\\
$12$ & $5$ & \pv{2040} & $1,20,24,152,511,628$\\
$12$ & $6$ & \pv{4283} & $1,19,100,431,874,1028$\\
$12$ & $7$ & \pv{8828} & $1,29,420,741,2727,3185$\\
$12$ & $8$ & \pv{16439} & $1,151,840,1278,2182,2913$\\
$12$ & $9$ & \pv{29308} & $1,219,1011,1509,6948,8506$\\
$12$ & $10$ & \pv{51154} & $1,39,1378,3775,5447,24629$\\
$13$ & $2$ & \mv{80} & $1,3,9,20,25,33$\\
$13$ & $3$ & \mv{312} & $1,14,74,77,130,138$\\
$13$ & $4$ & \pv{970} & $1,23,40,76,172,395$\\
$13$ & $5$ & \pv{2548} & $1,117,121,391,481,1101$\\
$13$ & $6$ & \pv{5598} & $1,12,216,450,1204,2708$\\
$13$ & $7$ & \pv{12176} & $1,45,454,1120,1632,1899$\\
$13$ & $8$ & \pv{22198} & $1,156,1166,2362,5999,9756$\\
$13$ & $9$ & \pv{40720} & $1,242,3091,4615,5162,13571$\\
$13$ & $10$ & \pv{72608} & $1,259,4815,8501,8623,23023$\\
$14$ & $2$ & \mv{90} & $1,4,10,17,26,29,41$\\
$14$ & $3$ & \mv{381} & $1,11,103,120,155,161,187$\\
$14$ & $4$ & \pv{1229} & $1,8,105,148,160,379,502$\\
$14$ & $5$ & \pv{3244} & $1,108,244,506,709,920,1252$\\
$14$ & $6$ & \pv{7815} & $1,197,460,696,975,2164,3032$\\
$14$ & $7$ & \pv{17389} & $1,123,955,1683,1772,2399,8362$\\
$14$ & $8$ & \pv{35929} & $1,796,1088,3082,3814,13947,14721$\\
$14$ & $9$ & \pv{71748} & $1,1223,3156,4147,5439,11841,25120$\\
$14$ & $10$ & \pv{126109} & $1,503,4548,7762,9210,9234,49414$\\
$15$ & $2$ & \mv{96} & $1,2,3,14,21,31,39$\\
$15$ & $3$ & \mv{448} & $1,10,127,150,176,189,217$\\
$15$ & $4$ & \pv{1420} & $1,20,111,196,264,340,343$\\
$15$ & $5$ & \pv{3980} & $1,264,300,382,668,774,1437$\\
$15$ & $6$ & \pv{9860} & $1,438,805,1131,1255,3041,3254$\\
$15$ & $7$ & \pv{22584} & $1,1396,2226,2309,2329,4582,9436$\\
$15$ & $8$ & \pv{48408} & $1,472,2421,3827,4885,5114,12628$\\
$15$ & $9$ & \pv{93804} & $1,3304,4679,9140,10144,10160,13845$\\
$15$ & $10$ & \pv{177302} & $1,2193,8578,18202,23704,23716,54925$\\
$16$ & $2$ & \mv{112} & $1,4,10,17,29,36,45,52$\\
$16$ & $3$ & \pv{518} & $1,8,36,46,75,133,183,247$\\
$16$ & $4$ & \pv{1717} & $1,46,144,272,297,480,582,601$\\
$16$ & $5$ & \pv{5024} & $1,380,451,811,1093,1202,1492,1677$\\
$16$ & $6$ & \pv{13380} & $1,395,567,1238,1420,1544,2526,4580$\\
$16$ & $7$ & \pv{32731} & $1,316,1150,1797,2909,4460,4836,16047$\\
$16$ & $8$ & \pv{71731} & $1,749,4314,7798,10918,11338,11471,25094$\\
$16$ & $9$ & \pv{148385} & $1,6094,6964,10683,11704,14274,14332,54076$\\
$16$ & $10$ & \pv{298105} & $1,5860,11313,15833,21207,26491,26722,99924$\\
$17$ & $2$ & \mv{130} & $1,7,26,37,47,49,52,61$\\
$18$ & $2$ & \mv{138} & $1,9,12,15,22,42,27,51,68$\\
$19$ & $2$ & \mv{156} & $1,15,21,23,26,33,52,61,65$\\
$20$ & $2$ & \mv{171} & $1,11,31,36,37,50,54,47,65,81$\\
$21$ & $2$ & \mv{192} & $1,3,15,23,32,51,57,64,85,91$\\
$22$ & $2$ & \mv{210} & $2,7,12,18,32,35,63,70,78,91,92$\\
$23$ & $2$ & \mv{216} & $1,3,5,17,27,36,43,57,72,83,95$\\
\hline
\end{tabular}
\captionsetup{justification=justified}
\caption{Largest circulant graphs of small degree $d$ and diameter $k$ found by computer search\\
* proven extremal}
\label{tab:search}
\end{table}

\bibliographystyle{amcjoucc}

\end{document}